\documentclass[reqno]{amsart}

\usepackage{tikz}\usetikzlibrary{matrix,arrows}
\usepackage[utf8]{inputenc}
\usepackage[T1]{fontenc}
\usepackage[mathcal]{euscript}
\usepackage{amssymb}

\theoremstyle{plain}
\newtheorem{theorem}{Theorem}
\newtheorem*{theorem*}{Theorem}
\newtheorem{corollary}[theorem]{Corollary}
\newtheorem*{corollary*}{Corollary}
\newtheorem{lemma}[theorem]{Lemma}
\newtheorem*{lemma*}{Lemma}

\newtheorem*{porism*}{Porism}
\newtheorem{proposition}[theorem]{Proposition}
\newtheorem*{proposition*}{Proposition}
\newtheorem{conjecture}[theorem]{Conjecture}
\newtheorem*{conjecture*}{Conjecture}

\theoremstyle{definition}
\newtheorem{definition}[theorem]{Definition}
\newtheorem*{definition*}{Definition}

\newtheorem*{example*}{Example}
\newtheorem{problem}[theorem]{Problem}
\newtheorem*{problem*}{Problem}

\theoremstyle{remark}

\newtheorem*{remark*}{Remark}

\newcommand{\n}{\mathbf{n}} 
\newcommand{\nc}{\mathrm{nc}} 
\newcommand{\fish}{\mathfrak{f}} 
\newcommand{\sym}{\mathcal{S}} 
\newcommand{\M}{\mathcal{M}} 
\newcommand{\N}{\mathcal{N}} 
\newcommand{\F}{\mathcal{F}} 
\newcommand{\T}{\mathcal{T}} 
\newcommand{\al}{\alpha}
\newcommand{\la}{\lambda}
\newcommand{\I}{\mathcal{I}} 
\newcommand{\ew}{\ensuremath{\epsilon}} 
\newcommand{\void}{\ensuremath{\emptyset}} 

\newcommand{\tpt}{\ensuremath{(\mathbf{2}+\mathbf{2})}}

\DeclareMathOperator{\lne}{lne} 
\DeclareMathOperator{\rne}{rne} 
\DeclareMathOperator{\lcr}{lcr} 
\DeclareMathOperator{\rcr}{rcr} 
\DeclareMathOperator{\inv}{inv}
\DeclareMathOperator{\asc}{asc}
\DeclareMathOperator{\des}{des}
\DeclareMathOperator{\comp}{comp}
\DeclareMathOperator{\lmin}{lmin}
\DeclareMathOperator{\rmin}{rmin}
\DeclareMathOperator{\lmax}{lmax}
\DeclareMathOperator{\rmax}{rmax}
\DeclareMathOperator{\last}{last}
\DeclareMathOperator{\emb}{emb}   
\DeclareMathOperator{\lev}{lev}   
\DeclareMathOperator{\dent}{dent} 
\DeclareMathOperator{\inter}{int} 
\DeclareMathOperator{\ip}{ip}     

\DeclareMathOperator{\Pre}{Pred} 
\DeclareMathOperator{\Suc}{Succ} 

\DeclareMathOperator{\pre}{pred} 
\DeclareMathOperator{\suc}{succ} 

\newcommand{\PATTERN}{
    \draw[step=1, xshift=14pt, yshift=14pt, \cfill, line cap=round] (0,0) grid (3,3);  
    \draw[step=1, xshift=14pt, yshift=14pt, thick] (0,1) -- (3,1);  
    \draw[step=1, xshift=14pt, yshift=14pt, thick] (1,0) -- (1,3);  
    \foreach \x/\y in {1/2,2/3,3/1} \node[disc, fill=black] at (\x,\y) {};  
}
\newcommand{\pattern}{\!\raisebox{-0.5em}{
  \begin{tikzpicture}[line width=0.7pt, scale=0.15]
    \tikzstyle{disc} = [circle,thin,draw=black, minimum size=1.7pt, inner sep=0pt ]
    \PATTERN
  \end{tikzpicture}}
}

\newcommand{\Pattern}{\raisebox{-0.5em}{
  \begin{tikzpicture}[line width=0.9pt, scale=0.3]
    \tikzstyle{disc} = [circle,thin,draw=black, minimum size=4pt, inner sep=0pt ]
    \PATTERN
  \end{tikzpicture}}
}

\newcommand{\cfill}{black!40}
\newcommand{\cfilll}{white}

\newcommand{\ns}{4pt}
\newcommand{\nodestyle}{\tikzstyle{every node} = [font=\footnotesize]}
\newcommand{\discstyle}{\tikzstyle{disc} = 
  [ circle,thin,fill=\cfilll,draw=black, minimum size=\ns, inner sep=0pt ] }
\newcommand{\style}{
  \nodestyle
  \discstyle
}

\newcommand{\pA}[2]{
  \begin{tikzpicture}[line width=#1, scale=#2]
    \style
    \node [disc] (r1)   at ( 0,0) {};
    \node [disc] (r11)  at ( 0,-1) {};
    \node [disc] (r111) at ( 0,-2) {};
    \draw 
    (r1) node[left=1pt] {3} -- (r11) node[left=1pt] {2} 
    -- (r111) node[left=1pt] {1}; 
  \end{tikzpicture}
}
\newcommand{\pB}[2]{
  \begin{tikzpicture}[line width=#1, scale=#2]
    \style
    \node [disc] (r1)  at (    0,-1) {};
    \node [disc] (r11) at (    0,-2) {};
    \node [disc] (r2)  at ( 0.75,-2) {};
    \draw (r1) node[above=1pt] {2} -- (r11) node[below=1pt] {1};
    \draw (r2) node[below=1pt] {3}; 
  \end{tikzpicture}
}
\newcommand{\pC}[2]{
  \begin{tikzpicture}[line width=#1, scale=#2]
    \style
    \node [disc] (r1)  at (    0,-1) {};
    \node [disc] (r11) at (    0,-2) {};
    \node [disc] (r2)  at ( 0.75,-2) {};
    \draw (r1) node[above=1pt] {3} -- (r11) node[below=1pt] {1};
    \draw (r2) node[below=1pt] {2}; 
  \end{tikzpicture}
}
\newcommand{\pD}[2]{
  \begin{tikzpicture}[line width=#1, scale=#2]
    \style
    \node [disc] (r1)  at ( 0,-1) {};
    \node [disc] (r11) at (-0.5,-2) {};
    \node [disc] (r12) at ( 0.5,-2) {};
    \draw 
    (r1) node[above=1pt] {3} -- (r11) node[below=1pt] {1} 
    (r1)                     -- (r12) node[below=1pt] {2}; 
  \end{tikzpicture}
}
\newcommand{\pE}[2]{
  \begin{tikzpicture}[line width=#1, scale=#2]
    \style
    \node [disc] (r1)  at ( 0,-2) {};
    \node [disc] (r11) at (-0.5,-1) {};
    \node [disc] (r12) at ( 0.5,-1) {};
    \draw 
    (r1) node[below=1pt] {1} -- (r11) node[above=1pt] {2} 
    (r1)                     -- (r12) node[above=1pt] {3}; 
  \end{tikzpicture}
}
\newcommand{\pF}[2]{
  \begin{tikzpicture}[line width=#1, scale=#2]
    \style
    \node [disc] (r1) at (-0.6,0) {};
    \node [disc] (r2) at (   0,0) {};
    \node [disc] (r3) at ( 0.6,0) {};
    \draw (r1) node[below=1pt] {1};
    \draw (r2) node[below=1pt] {2};
    \draw (r3) node[below=1pt] {3};
  \end{tikzpicture}
}

\title{$n!$ matchings, $n!$ posets}
\author[A. Claesson]{Anders Claesson}
\address{The Mathematics Institute, School of Computer Science,
Reykjavik University, Menntavegi 1, 101 Reykjavik, Iceland}
\author[S. Linusson]{Svante Linusson}
\address{Department of Mathematics, KTH-Royal Institute of Technology,
SE-100 44 Stockholm, Sweden}
\thanks{Claesson was supported by grant no. 090038011 
  from the Icelandic Research Fund. Linusson is a Royal Swedish Academy of Sciences 
  Research Fellow supported by a grant from the Knut and Alice Wallenberg 
  Foundation.
}

\begin{document}

\begin{abstract}
  We show that there are $n!$ matchings on $2n$ points without, so
  called, left (neighbor) nestings. We also define a set of naturally
  labeled $\tpt$-free posets, and show that there are $n!$ such posets
  on $n$ elements. Our work was inspired by 
  Bousquet-M\'elou,
  Claesson, Dukes and Kitaev [J. Combin. Theory Ser. A. 117 (2010) 884--909]. 
  They gave bijections
  between four classes of combinatorial objects: matchings with no
  neighbor nestings (due to Stoimenow), unlabeled $\tpt$-free posets,
  permutations avoiding a specific pattern, and so called ascent
  sequences. We believe that certain statistics on our matchings and
  posets could generalize the work of Bousquet-M\'elou et al.\ and we
  make a conjecture to that effect. We also identify natural subsets
  of matchings and posets that are equinumerous to the class of
  unlabeled $\tpt$-free posets.
  
  We give bijections that show the equivalence of (neighbor)
  restrictions on nesting arcs with (neighbor) restrictions on
  crossing arcs. These bijections are thought to be of independent
  interest. One of the bijections maps via certain upper-triangular
  integer matrices that have recently been studied by Dukes and
  Parviainen [Electron. J. Combin. 17 (2010) \#R53]
\end{abstract}

\maketitle\thispagestyle{empty}

\section{Introduction}

A \emph{matching} of the integers $\{1,2,\ldots , 2n\}$ is a
partition of that set into blocks of size 2.
An example of a matching is
$$M=\{(1,3), (2,7), (4,6), (5,8)\}.
$$ 
In the diagram below there is an \emph{arc} connecting $i$ with
$j$ precisely when $(i,j)\in M$.
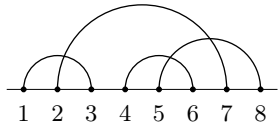
\medskip
\begin{center}
  \begin{tikzpicture}[line width=.5pt, scale=0.45, font=\small]  
  \filldraw   (1,0) circle (1.9pt); \node at (1,-0.7) {1};
  \filldraw   (2,0) circle (1.9pt); \node at (2,-0.7) {2};
  \filldraw   (3,0) circle (1.9pt); \node at (3,-0.7) {3};
  \filldraw   (4,0) circle (1.9pt); \node at (4,-0.7) {4};
  \filldraw   (5,0) circle (1.9pt); \node at (5,-0.7) {5};
  \filldraw   (6,0) circle (1.9pt); \node at (6,-0.7) {6};
  \filldraw   (7,0) circle (1.9pt); \node at (7,-0.7) {7};
  \filldraw   (8,0) circle (1.9pt); \node at (8,-0.7) {8};
   \draw[very thin] (0.5,0) -- (8.500000,0);
  \draw (3,0) arc (0:180:1.0000);
  \draw (6,0) arc (0:180:1.0000);
  \draw (7,0) arc (0:180:2.5000);
  \draw (8,0) arc (0:180:1.5000);
  \end{tikzpicture}\vspace{-0.5ex}
\end{center}
A \emph{nesting} of $M$ is a pair of arcs
$(i,\ell)$ and $(j,k)$ with $i<j<k<\ell$:
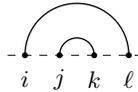
\medskip
\begin{center}
  \begin{tikzpicture}[line width=.5pt, scale=0.46, font=\small]
    \filldraw (1,0) circle (1.9pt); 
    \filldraw (2,0) circle (1.9pt); 
    \filldraw (3,0) circle (1.9pt); 
    \filldraw (4,0) circle (1.9pt); 
    \node at (1,-0.7) {$i$};
    \node at (2,-0.7) {$j$};
    \node at (3,-0.7) {$k$};
    \node at (4,-0.7) {$\ell$};
    \draw[very thin,dashed] (0.5,0) -- (4.5,0);
    \draw (3,0) arc (0:180:0.5000);
    \draw (4,0) arc (0:180:1.5000);
  \end{tikzpicture}\vspace{-0.6ex}
\end{center}
We call such a nesting a \emph{left-nesting} if $j=i+1$. Similarly, we
call it a \emph{right-nesting} if $\ell=k+1$. The example matching has
one nesting, formed by the two arcs $(2,7)$ and $(4,6)$. It is a
right-nesting.

To give upper bounds on the dimension of the space of Vassiliev's knot
invariants of a given degree, Stoimenow~\cite{stoim} was led to
introduce what he calls regular linearized chord diagrams. In the
terminology of this paper, Stoimenow's diagrams are matchings with no
\emph{neighbor nestings}, that is, matchings with neither
left-nestings, nor right-nestings. Following Stoimenow's paper,
Zagier~\cite{zagier} derived the following beautiful generating
function enumerating such matchings with respect to size:
$$\sum_{n\ge 0}\prod_{i=1}^{n} \left( 1-(1-t)^i\right).
$$

Recently, Bousquet-M\'elou et al.~\cite{BCDK} gave bijections between
matchings on $[2n]$ with no neighbor nestings and three other classes
of combinatorial objects, thus proving that they are equinumerous. The
other classes were unlabeled \tpt-free posets (or interval orders) on
$n$ nodes; permutations on $[n]$ avoiding the pattern \mbox{\pattern;} and
ascent sequences of length $n$. Let $\fish_n$ be the cardinality of
any, and thus all, of the above classes---it is the coefficient in
front of $t^n$ in Zagier's generating function. We call $\fish_n$ the
$n$th \emph{Fishburn number}; the sequence starts 
$$1, 1, 2, 5, 15, 53, 217, 1014, 5335, 31240,\dots
$$
Fishburn~\cite{fishburn-book, fishburn,
  fishburn2} did pioneering work on interval orders; for instance, he
showed the basic theorem that a poset is an interval order if and only
if it is {\tpt}-free.

The pattern avoiding permutations and the ascent sequences were both
defined by Bousquet-M\'elou et al. We shall recall those definitions
here. In a permutation $\pi=a_1\dots a_n$, an occurrence of the
pattern
$$\Pattern
$$
is a subsequence $a_ia_{i+1}a_j$ such that
$a_j+1=a_i<a_{i+1}$. As an example, the permutation $\pi = 351426$
contains one such occurrence:
$$
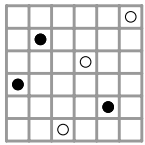
\begin{tikzpicture}[line width=1.1pt, scale=0.3]
  \style
  \draw[step=1, xshift=14pt, yshift=14pt, \cfill, line cap=round]
  (0,0) grid (6,6);  
  \foreach \x/\y in {1/3,2/5,3/1,4/4,5/2,6/6}
  \node[disc] (\y) at (\x,\y) {};  
  \foreach \y in {3,5,2}
  \node[disc,fill=black] at (\y) {};  
\end{tikzpicture}
$$ If $\pi$ contains no such occurrence we say that $\pi$
\emph{avoids} the pattern. An integer sequence $(x_1,\dots,x_n)$ is an
\emph{ascent sequence} if
$$x_1=0\quad\text{and}\quad 0\leq x_i \leq 1 + \asc(x_1,\dots,x_{i-1}),
$$ for $2\leq i\leq n$. Here, $\asc(x_1,\dots, x_k)$ denotes the
number of ascents in $(x_1,\dots, x_k)$, and an ascent is a $j\in
[k-1]$ such that $x_j < x_{j+1}$. Bousquet-M\'elou et al.~\cite{BCDK}
derived a closed expression for the generating function enumerating
ascent sequences with respect to length and number of ascents; hence
they gave a new proof of Zagier's result, or rather a refinement of
it.

Recall that Stoimenow's diagrams are matchings with no neighbor
nestings. The discovery that led to the present paper is that there
are exactly $n!$ matchings on $[2n]$ with no left-nestings
(Theorem~\ref{thm:no-left-nestings}). As an example, these are the $6$
such matchings on $\{1,\dots,6\}$:
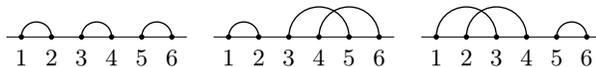
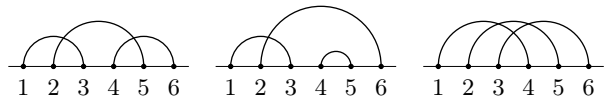
$$
  \begin{tikzpicture}[line width=.5pt, scale=0.4, font=\small]
  
  \filldraw   (1,0) circle (1.9pt); \node at (1,-0.7) {1};
  \filldraw   (2,0) circle (1.9pt); \node at (2,-0.7) {2};
  \filldraw   (3,0) circle (1.9pt); \node at (3,-0.7) {3};
  \filldraw   (4,0) circle (1.9pt); \node at (4,-0.7) {4};
  \filldraw   (5,0) circle (1.9pt); \node at (5,-0.7) {5};
  \filldraw   (6,0) circle (1.9pt); \node at (6,-0.7) {6};
   \draw[very thin] (0.5,0) -- (6.500000,0);
  \draw (2,0) arc (0:180:0.5000);
  \draw (4,0) arc (0:180:0.5000);
  \draw (6,0) arc (0:180:0.5000);
  \end{tikzpicture}\quad  
  \begin{tikzpicture}[line width=.5pt, scale=0.4, font=\small]
  
  \filldraw   (1,0) circle (1.9pt); \node at (1,-0.7) {1};
  \filldraw   (2,0) circle (1.9pt); \node at (2,-0.7) {2};
  \filldraw   (3,0) circle (1.9pt); \node at (3,-0.7) {3};
  \filldraw   (4,0) circle (1.9pt); \node at (4,-0.7) {4};
  \filldraw   (5,0) circle (1.9pt); \node at (5,-0.7) {5};
  \filldraw   (6,0) circle (1.9pt); \node at (6,-0.7) {6};
   \draw[very thin] (0.5,0) -- (6.500000,0);
  \draw (2,0) arc (0:180:0.5000);
  \draw (5,0) arc (0:180:1.0000);
  \draw (6,0) arc (0:180:1.0000);
  \end{tikzpicture}\quad
  \begin{tikzpicture}[line width=.5pt, scale=0.4, font=\small]
  
  \filldraw   (1,0) circle (1.9pt); \node at (1,-0.7) {1};
  \filldraw   (2,0) circle (1.9pt); \node at (2,-0.7) {2};
  \filldraw   (3,0) circle (1.9pt); \node at (3,-0.7) {3};
  \filldraw   (4,0) circle (1.9pt); \node at (4,-0.7) {4};
  \filldraw   (5,0) circle (1.9pt); \node at (5,-0.7) {5};
  \filldraw   (6,0) circle (1.9pt); \node at (6,-0.7) {6};
   \draw[very thin] (0.5,0) -- (6.500000,0);
  \draw (3,0) arc (0:180:1.0000);
  \draw (4,0) arc (0:180:1.0000);
  \draw (6,0) arc (0:180:0.5000);
  \end{tikzpicture}
  $$
  $$
  \begin{tikzpicture}[line width=.5pt, scale=0.4, font=\small]
  
  \filldraw   (1,0) circle (1.9pt); \node at (1,-0.7) {1};
  \filldraw   (2,0) circle (1.9pt); \node at (2,-0.7) {2};
  \filldraw   (3,0) circle (1.9pt); \node at (3,-0.7) {3};
  \filldraw   (4,0) circle (1.9pt); \node at (4,-0.7) {4};
  \filldraw   (5,0) circle (1.9pt); \node at (5,-0.7) {5};
  \filldraw   (6,0) circle (1.9pt); \node at (6,-0.7) {6};
   \draw[very thin] (0.5,0) -- (6.500000,0);
  \draw (3,0) arc (0:180:1.0000);
  \draw (5,0) arc (0:180:1.5000);
  \draw (6,0) arc (0:180:1.0000);
  \end{tikzpicture}\quad
  \begin{tikzpicture}[line width=.5pt, scale=0.4, font=\small]
  
  \filldraw   (1,0) circle (1.9pt); \node at (1,-0.7) {1};
  \filldraw   (2,0) circle (1.9pt); \node at (2,-0.7) {2};
  \filldraw   (3,0) circle (1.9pt); \node at (3,-0.7) {3};
  \filldraw   (4,0) circle (1.9pt); \node at (4,-0.7) {4};
  \filldraw   (5,0) circle (1.9pt); \node at (5,-0.7) {5};
  \filldraw   (6,0) circle (1.9pt); \node at (6,-0.7) {6};
   \draw[very thin] (0.5,0) -- (6.500000,0);
  \draw (3,0) arc (0:180:1.0000);
  \draw (5,0) arc (0:180:0.5000);
  \draw (6,0) arc (0:180:2.0000);
  \end{tikzpicture}\quad  
  \begin{tikzpicture}[line width=.5pt, scale=0.4, font=\small]
  
  \filldraw   (1,0) circle (1.9pt); \node at (1,-0.7) {1};
  \filldraw   (2,0) circle (1.9pt); \node at (2,-0.7) {2};
  \filldraw   (3,0) circle (1.9pt); \node at (3,-0.7) {3};
  \filldraw   (4,0) circle (1.9pt); \node at (4,-0.7) {4};
  \filldraw   (5,0) circle (1.9pt); \node at (5,-0.7) {5};
  \filldraw   (6,0) circle (1.9pt); \node at (6,-0.7) {6};
   \draw[very thin] (0.5,0) -- (6.500000,0);
  \draw (4,0) arc (0:180:1.5000);
  \draw (5,0) arc (0:180:1.5000);
  \draw (6,0) arc (0:180:1.5000);
  \end{tikzpicture}
$$

Can we also ``lift'' ascent sequences and unlabeled \tpt-free posets
to the level of all permutations? That is, can we define ``certain
sequences'' and ``certain posets'', both of cardinality $n!$, that are
supersets of ascent sequences and unlabeled \tpt-free posets,
respectively? For ascent sequences this is easy, and inversion tables
is a natural choice. The poset case is more challenging. However, we
show (Definition~\ref{def:factorial-poset} and
Theorem~\ref{thm:factorial-poset}) that there are exactly $n!$
naturally labeled posets $P$ on $[n]$ such that $i<_P k$ whenever
$i<j<_Pk$ for some $j\in[n]$; we call them \emph{factorial
  posets}. Here is a list of the 6 factorial posets on $\{1,2,3\}$:
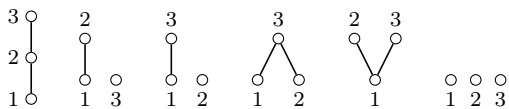
$$
\pA{0.6pt}{0.55}\quad\;
\pB{0.6pt}{0.55}\quad
\pC{0.6pt}{0.55}\quad
\pD{0.6pt}{0.55}\quad
\pE{0.6pt}{0.55}\quad
\pF{0.6pt}{0.55}\label{6-factorial-posets}
$$ 
It is not hard to see
(Proposition~\ref{prop:factorial-posets-are-2+2-free}) that factorial
posets are \tpt-free. Moreover, we give an additional restriction on
the labeling of factorial posets under which the labeling is unique
(Proposition~\ref{prop:2+2-free}), and thus the subset of factorial
posets meeting that restriction is trivially in bijection with
unlabeled \tpt-free posets.

The bijections we give to prove that inversion tables, factorial
posets and matchings with no left-nesting are equinumerous do however
not specialize to give back the results from \cite{BCDK}. This remains
an interesting challenge.  In Section \ref{sec:cross-vs-nest} we prove
that we could have studied matchings with restrictions on crossings
instead of on nestings and present bijections to verify this.

Let $p=\pattern$. As mentioned before, Bousquet-M\'elou et
al.~\cite{BCDK} gave a bijection between matchings with no neighbor
nestings and $p$-avoiding permutations. We conjecture
(Conjecture~\ref{conj:distribution-of-rne}) a generalization of that
result. Namely, we conjecture that the distribution of right-nestings
over matchings on $[2n]$ with no left-nestings coincides with the
distribution of $p$ over permutations on $[n]$.

In a recent paper, Dukes and Parviainen~\cite{DP} study upper
triangular matrices with non-negative integer entries such that each
row and column has at least one nonzero entry and the total sum of the
entries is $n$. They provide a recursive encoding of those matrices as
ascent sequences. We have found a direct bijection
(Theorem~\ref{thm:no-neighbor-nestings}) from the same matrices to
matchings with no neighbor nestings. In addition, we show
(Proposition~\ref{prop:01-matrices}) that the subset of the matrices
whose entries are 0 or 1 are in bijection with matchings with no
left-nestings and no right-crossings.

\section{Matchings with no left-nestings}

Let $\M_n$ be the set of matchings on $[2n]$, and let $M\in\M_n$. If
$i<j$ and $\al=(i,j)$ is an arc of $M$ we call $i$ the \emph{opener}
of $\al$, and we call $j$ the \emph{closer} of $\al$.  In what follows
it will be convenient to order the arcs with respect to closer. In
particular, ``the last arc'' refers to the arc with closer $2n$. In
the introduction we defined what left- and right-nestings are, and by
$\lne(M)$ and $\rne(M)$ we shall denote the number of left- and
right-nestings, respectively. Let
$$\N_n = \{\,M\in\M_n : \lne(M)=0\,\}
$$
and $\N=\cup_{n\geq 0}\N_n$.
Define $\I_n$ as the Cartesian product
$$\I_n=[0,0]\times [0,1]\times\dots\times [0,n-1],
$$ where $[i,j]=\{i,i+1,\dots, j\}$. In other words, $\I_n$ is the set
of inversion tables of length $n$. Also, let $\I=\cup_{n\geq 0}\I_n$.

\begin{theorem}\label{thm:no-left-nestings}
  Matchings of \/$[2n]$\! with no left-nestings are in bijection with
  inversion tables of length $n$, and thus $|\N_n|=n!$.
\end{theorem}

\begin{proof}
  Using recursion we define a bijection $f:\I\to\N$. Let $f(\ew) =
  \emptyset$, that is, let the empty inversion table map to the empty
  matching. Let $w=(a_1,\dots, a_n)$ be any inversion table in $\I_n$
  with $n>0$. Let $w'=(a_1,\dots,a_{n-1})$ and let $M'=f(w')$. Now
  create a matching $M$ in $\N_n$ by inserting a new last arc in $M'$
  whose opener is immediately to the left of the $(a_n+1)$st closer
  of $M'$ if $a_n<n-1$ and immediately to the left of its own closer
  if $a_n=n-1$.  Set $f(w)=M$. Note that the opener of the last arc
  has to be immediately to the left of some closer, otherwise a
  left-nesting would be created. Also note that removing the last arc
  from a matching in $\N_n$ cannot create a left-nesting. From a
  simple induction argument it thus follows that the described map is
  a bijection.

  It is also easy to give a direct, non-recursive, description of the
  inverse of $f$. Indeed, $f^{-1}(M)=(a_1,\dots,a_n)$ where $a_i$ is
  the number of closers to the left of the opener of the $i$th arc;
  here, as before, arcs are ordered by closer.
\end{proof}

As an example, let $w=(a_1,a_2,a_3,a_4)=(0,1,0,1)$. To construct the
matching corresponding to that inversion table we insert the arcs one
at the time, so that---as in the proof---the opener of the new last
arc is immediately to the left of the $(a_i+1)$st closer:
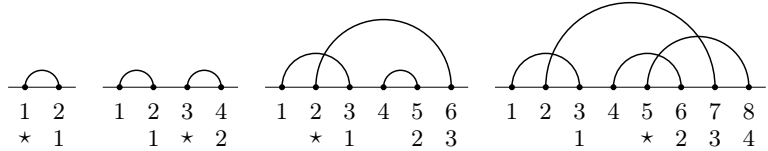
\begin{center}
  \begin{tikzpicture}[line width=.56pt, scale=0.45, font=\small]
  
  \filldraw   (1,0) circle (1.9pt); \node at (1,-0.7) {1};
  \filldraw   (2,0) circle (1.9pt); \node at (2,-0.7) {2};
   \draw[very thin] (0.5,0) -- (2.500000,0);
  \draw (2,0) arc (0:180:0.5000);
   \node at (2,-1.5) {1};
   \node at (1,-1.5) {$\star$};
  \end{tikzpicture}\quad
  \begin{tikzpicture}[line width=.56pt, scale=0.45, font=\small]
  
  \filldraw   (1,0) circle (1.9pt); \node at (1,-0.7) {1};
  \filldraw   (2,0) circle (1.9pt); \node at (2,-0.7) {2};
  \filldraw   (3,0) circle (1.9pt); \node at (3,-0.7) {3};
  \filldraw   (4,0) circle (1.9pt); \node at (4,-0.7) {4};
   \draw[very thin] (0.5,0) -- (4.500000,0);
  \draw (2,0) arc (0:180:0.5000);
  \draw (4,0) arc (0:180:0.5000);
   \node at (2,-1.5) {1};
   \node at (4,-1.5) {2};
   \node at (3,-1.5) {$\star$};
  \end{tikzpicture}\quad
  \begin{tikzpicture}[line width=.56pt, scale=0.45, font=\small]
  
  \filldraw   (1,0) circle (1.9pt); \node at (1,-0.7) {1};
  \filldraw   (2,0) circle (1.9pt); \node at (2,-0.7) {2};
  \filldraw   (3,0) circle (1.9pt); \node at (3,-0.7) {3};
  \filldraw   (4,0) circle (1.9pt); \node at (4,-0.7) {4};
  \filldraw   (5,0) circle (1.9pt); \node at (5,-0.7) {5};
  \filldraw   (6,0) circle (1.9pt); \node at (6,-0.7) {6};
   \draw[very thin] (0.5,0) -- (6.500000,0);
  \draw (3,0) arc (0:180:1.0000);
  \draw (5,0) arc (0:180:0.5000);
  \draw (6,0) arc (0:180:2.0000);
   \node at (3,-1.5) {1};
   \node at (5,-1.5) {2};
   \node at (6,-1.5) {3};
   \node at (2,-1.5) {$\star$};
  \end{tikzpicture}\quad
  \begin{tikzpicture}[line width=.56pt, scale=0.45, font=\small]
  
  \filldraw   (1,0) circle (1.9pt); \node at (1,-0.7) {1};
  \filldraw   (2,0) circle (1.9pt); \node at (2,-0.7) {2};
  \filldraw   (3,0) circle (1.9pt); \node at (3,-0.7) {3};
  \filldraw   (4,0) circle (1.9pt); \node at (4,-0.7) {4};
  \filldraw   (5,0) circle (1.9pt); \node at (5,-0.7) {5};
  \filldraw   (6,0) circle (1.9pt); \node at (6,-0.7) {6};
  \filldraw   (7,0) circle (1.9pt); \node at (7,-0.7) {7};
  \filldraw   (8,0) circle (1.9pt); \node at (8,-0.7) {8};
   \draw[very thin] (0.5,0) -- (8.500000,0);
  \draw (3,0) arc (0:180:1.0000);
  \draw (6,0) arc (0:180:1.0000);
  \draw (7,0) arc (0:180:2.5000);
  \draw (8,0) arc (0:180:1.5000);
   \node at (3,-1.5) {1};
   \node at (6,-1.5) {2};
   \node at (7,-1.5) {3};
   \node at (8,-1.5) {4};
   \node at (5,-1.5) {$\star$};
  \end{tikzpicture}
\end{center}
Here the star marks the opener of the new arc. Reading the number to
the right of the star we get $(1,2,1,2)$ and subtracting one from each
coordinate we recover the inversion table $(0,1,0,1)$.

\section{Factorial posets}

A poset $P$ of cardinality $n$ is said to be labeled if its elements
are identified with the integers $1,\dots,n$. A poset $P$ is
\emph{naturally labeled} if $i<j$ in $P$ implies $i<j$ in the usual
order.

\begin{definition}\label{def:factorial-poset}
  We call a naturally labeled poset $P$ on $[n]$ such that, for $i,j,k\in [n]$,
  $$ i<j<_Pk\implies i<_P k
  $$ a \emph{factorial poset}, and by $\F_n$ we denote the set of
  factorial posets on $[n]$. Similarly, we call a naturally labeled
  poset $P$ on $[n]$ such that, for $i,j,k\in [n]$,
  $$ i>j>_P k\implies i>_P k
  $$ a \emph{dually factorial poset}.
\end{definition}

There are 6 factorial posets on $\{1,2,3\}$, and we listed them on
page~\ref{6-factorial-posets}.  It is easy to check that of those
posets, exactly one is not dually factorial, namely
$$  
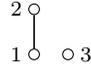
\begin{tikzpicture}[line width=0.6pt, scale=0.6]
  \style
  \node [disc] (r1)  at (    0,-1) {};
  \node [disc] (r11) at (    0,-2) {};
  \node [disc] (r2)  at ( 0.75,-2) {};
  \draw (r1) node[left=1pt] {2} -- (r11) node[left=1pt] {1};
  \draw (r2) node[right=1pt] {3}; 
\end{tikzpicture}\vspace{-3pt}
$$
Denoting this poset by $P$ we have $3>2>_P1$, but $3\not >_P1$.
\begin{definition}
  The \emph{predecessor set} of $j\in P$ is $\Pre(j) = \{ i : i <_P j
  \}$, and we denote by $\pre(j)=\#\Pre(j)$ the number of predecessors
  of $j$. Similarly we define $\Suc(j) = \{ i : i >_P j \}$ as the
  \emph{successor set} of $j$ and $\suc(j)=\#\Suc(j)$ as the number of
  successors of $j$.
\end{definition}
Note that $P$ is factorial if, and only if, for all $k$ in $P$, there
is a $j$ in $[0,n-1]$, such that $\Pre(k)=[1,j]$.  It is well
known---see for example Bogart~\cite{bogart}---that a poset is
{\tpt}-free if, and only if, the collection $\{ \,\Pre(k) : k\in
P\,\}$ of predecessor sets can be linearly ordered by inclusion; hence
the following proposition.

\begin{proposition}\label{prop:factorial-posets-are-2+2-free}
  Factorial posets are {\tpt}-free.
\end{proposition}

\begin{theorem}\label{thm:factorial-poset}
  Factorial posets on $[n]$ are in bijection with inversion tables of
  length $n$, and thus $|\F_n|=n!$.
\end{theorem}

\begin{proof}
  Define $g:\F_n\to\I_n$ by $g(P) = (a_1,\dots,a_n)$ where
  $a_k=\pre(k)$. To see that $g$ is a bijection we describe its
  inverse. Given an inversion table $w=(a_1,\dots,a_n)$ in $\I_n$ we
  construct a factorial poset $P=P(w)$ by postulating that $i<_P k$
  precisely when $1\leq i\leq a_k$. That this definition is consistent
  is easily seen by building $P$ recursively.
\end{proof}

We now have two bijections, $f$ from inversion tables to matchings
with no left-nestings, and $g$ from factorial posets to
inversion tables. Let $h=f\circ g$ be their composition:
$$
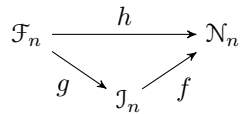
\begin{tikzpicture}[>=stealth']
  \matrix (m) [matrix of math nodes, row sep=1em, column sep=2.1em]
    { \F_n  &          &  \N_n \\
            & \I_n\!\! &       \\
    };
   \path[->]
   (m-1-1) edge node[below left=0pt]   {$g$} (m-2-2)
   (m-2-2) edge node[below right=-1pt] {$f$} (m-1-3)
   (m-1-1) edge node[above=0.1pt]      {$h$} (m-1-3)
   ;
\end{tikzpicture}
$$ Let $P\in\F_n$. From the proofs of
Theorems~\ref{thm:factorial-poset} and \ref{thm:no-left-nestings} it
is immediate that to build $M=h(P)$ we insert the arcs one at the time
so that, in the $i$th step, the opener of the new last arc is
immediately to the left of the $(\pre(i)+1)$st closer.

Next we describe the inverse map, $h^{-1}$. Take $M\in\N_n$ and let
$\al_1,\dots,\al_n$ be its arcs ordered by closer. Then $i<j$ in
$P=h^{-1}(M)$ if and only if the closer of $\al_i$ is to the left of
the opener of $\al_j$.

An \emph{interval order} is a poset with the property that each
element $x$ can be assigned an interval $I(x)$ of real numbers so that
$x<y$ in the poset if and only if every point in $I(x)$ is less than
every point in $I(y)$. Such an assignment is called an \emph{interval
  representation} of the poset. In 1970, Fishburn~\cite{fishburn}
showed that a poset is $\tpt$-free precisely when it has an interval
representation. Let us for a moment identify the arcs of a matching
with intervals of the real line. Then the function $h$, above, gives
an interval representation of each factorial poset.

\section{A unique labeling}

Let $M\in\N_n$ and let $\al_1,\dots,\al_n$ be its arcs ordered by
closer. Let $P=h^{-1}(M)$. Assume that $1\leq i<j\leq n$ in the usual
order. Note that if $\al_i$ and $\al_j$ form a nesting, then we cannot
have $\pre(i)=\pre(j)$ since then it would be a left-nesting which can
never occur by the definition of $g^{-1}$.  Thus $\al_i$ and $\al_j$
form a nesting precisely when $\pre(i)>\pre(j)$. If, in addition,
$j=i+1$ and $\suc(i)=\suc(j)$ then $\al_i$ and $\al_j$ form a
right-nesting. Thus $M$ is non-neighbor-nesting precisely when for
each $i\in [n-1]$ we have $\pre(i)\le\pre(i+1)$ or
$\suc(i)>\suc(i+1)$. By applying the bijection of Bousquet-M\'elou et
al.~\cite{BCDK} from non-neighbor-nesting matchings to unlabeled
$\tpt$-free posets we get the following result.

\begin{proposition}\label{prop:2+2-free}
  Factorial posets on $[n]$ such that for each $i\in [n-1]$ we have
  \begin{equation}\label{1}
    \pre(i)\le \pre(i+1)\;\text{ or }\;\suc(i)>\suc(i+1)
  \end{equation} 
  are in bijection with unlabeled \tpt-free posets on $n$ nodes; hence
  there are exactly $\fish_n$ such posets.
\end{proposition}

An alternative way to see the above result is that given an unlabeled
\tpt-free poset $P$ there is exactly one way to label $P$ so that the
resulting poset is factorial and satisfies \eqref{1}. The key
observation to such a labeling is that if $P$ is factorial and
$\eqref{1}$ holds then the pairs $(\suc(1),\pre(1))$, \dots,
$(\suc(n),\pre(n))$ are ordered weakly decreasing with respect to the
first coordinate, and on equal first coordinate weakly increasing with
respect to the second coordinate. As an example we consider the
unlabeled \tpt-free poset
$$P\,=
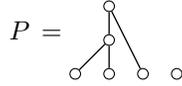
\begin{tikzpicture}[line width=.5pt, scale=.45, baseline=(c.base)]
  \style
  \node [disc] (d) at ( 0.5, 0) {};
  \node [disc] (c) at ( 0.5,-1) {};
  \node [disc] (e) at ( 1.5,-2) {};
  \node [disc] (a) at (-0.5,-2) {};
  \node [disc] (b) at ( 0.5,-2) {};
  \node [disc] (f) at ( 2.5,-2) {};
  \draw (d) -- (c) (d) -- (e) (c) -- (a) (c) -- (b);
\end{tikzpicture}\smallskip
$$ We shall use the observation above to label $P$ so that it is
factorial and satisfies~$\eqref{1}$. We start by naming the vertices
$a$, $b$, $c$, $d$, $e$ and $f$. Then we calculate predecessor and
successor sets:
$$
\begin{tikzpicture}[line width=.5pt, scale=.55, baseline=(c.base)]
  \style
  \node [disc] (d) at ( 0.5, 0) {};
  \node [disc] (c) at ( 0.5,-1) {};
  \node [disc] (e) at ( 1.5,-2) {};
  \node [disc] (a) at (-0.5,-2) {};
  \node [disc] (b) at ( 0.5,-2) {};
  \node [disc] (f) at ( 2.5,-2) {};
  \draw 
  (d) node[above=1pt] {$d$} -- (c) node[above left=-0.2pt] {$c$} 
  (d)                       -- (e) node[below=1.6pt] {$e$} 
  (c)                       -- (a) node[below=1.6pt] {$a$}
  (c)                       -- (b) node[below=1pt] {$b$}
  (f) node[below=1pt] {$f$}; 
\end{tikzpicture}\;
\begin{array}{cccccccc}
 x                \!\!&=&\!\! a     & b     & c     & d     & e     & f    \\
 \Pre(x)          \!\!&=&\!\!\void  &\void  & ab    & abce  &\void  &\void \\
 \Suc(x)          \!\!&=&\!\! cd    & cd    & d     &\void  & d     &\void \\
 (\suc(x),\pre(x))\!\!&=&\!\! (2,0) & (2,0) & (1,2) & (0,4) & (1,0) & (0,0)\\
 \text{label}     \!\!&=&\!\! 1     & 2     & 4     & 6     & 3     & 5   
\end{array}
$$ Finally, we label the elements with the integers $1$, $2$, $3$, $4$
and $5$ so that when we read the $(\suc,\pre)$-pair for the vertex
labeled $1$, and then the $(\suc,\pre)$-pair for the vertex labeled
$2$, and so on, then those pairs are read in the prescribed order:
$$
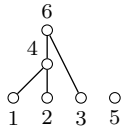
\begin{tikzpicture}[line width=.5pt, scale=.45, baseline=(c.base)]
  \style
  \node [disc] (d) at ( 0.5, 0) {};
  \node [disc] (c) at ( 0.5,-1) {};
  \node [disc] (e) at ( 1.5,-2) {};
  \node [disc] (a) at (-0.5,-2) {};
  \node [disc] (b) at ( 0.5,-2) {};
  \node [disc] (f) at ( 2.5,-2) {};
  \draw 
  (d) node[above=1pt] {$6$} -- (c) node[above left=-0.2pt] {$4$} 
  (d)                       -- (e) node[below=1.6pt] {$3$} 
  (c)                       -- (a) node[below=1.6pt] {$1$}
  (c)                       -- (b) node[below=1pt] {$2$}
  (f) node[below=1pt] {$5$}; 
\end{tikzpicture}\smallskip
$$ 
In this example, we could have chosen to label $a$ by 2 and $b$ by
1; it would not have made a difference since the vertices named $a$ and
$b$ are indistinguishable.

\section{Crossings versus nestings}\label{sec:cross-vs-nest}

A \emph{crossing} of a matching $M$ is a pair of arcs $(i,k)$ and
$(j,\ell)$ with $i<j<k<\ell$, and we can define left- and
right-crossings analogously to how it was defined for nesting
arcs. With $A$ and $B$ as in the table below there are bijections
between 
$$
\{M\in\M_n : \text{$M$ is non $A$}\}\,\text{ and }\,
\{M\in\M_n : \text{$M$ is non $B$}\}.
$$
\begin{center}
  \begin{tabular}{ccc}
    A                && B \\
    nesting          && crossing \\
    neighbor nesting && neighbor crossing \\
    left-nesting     && left-crossing
  \end{tabular}
\end{center}

The first case is well known: for bijections between non-nesting
matchings and non-crossing matchings see for instance \cite{deM,deSC,
  KaZe}. We give bijections for the two remaining cases in this
section. There exists a more complicated bijection~\cite{sund} that
can explain all three levels at once; see comment at the end of this
section.

The second case is the most challenging, so let us look at the third
case first. The proof of Theorem~\ref{thm:no-left-nestings} gives a
bijection $f$ from inversion tables to non-left-nesting
matchings. That bijection can be modified to give a bijection
$f_{\nc}$ from inversion tables to non-left-crossing matchings
(Theorem~\ref{thm:no-left-crossings}), and so $f_{\nc}\circ f^{-1}$ is
a bijection from non-left-nesting to non-left-crossing matchings.

\begin{theorem}\label{thm:no-left-crossings}
  Matchings of \/$[2n]$\! with no left-crossing are in bijection with
  inversion tables of length $n$; hence there are exactly $n!$ such
  matchings.
\end{theorem}

\begin{proof}
  As in the proof of Theorem~\ref{thm:no-left-nestings} we
  define a bijection $f_{\nc}$ recursively. The difference is that
  this time the opener of the new last arc is immediately to the right
  of the $a_n$th closer if $a_n>0$, or to the extreme left if
  $a_n=0$.
\end{proof}

For the second case, we shall give a bijection via matrices of a
certain kind. Let $\T_n$ be the set of upper triangular matrices with
non-negative integer entries, such that no row or column has only
zeros and the total sum of the entries is $n$. These matrices have
recently been studied by Dukes and Parviainen~\cite[\S2]{DP}; they
gave a recursive encoding of the matrices in $\T_n$ as ascent
sequences, and thus they showed that $|\T_n|=\fish_n$. This fact seems
to have been first observed by Vladeta Jovovic~\cite{sloane}. We
shall give a surjection $\psi$ from the set of matchings of $[2n]$ to
$\T_n$. Further, we shall show that if $\psi$ is restricted to
non-neighbor-nesting matchings, or non-neighbor-crossing matchings,
then $\psi$ is a bijection.

Before we describe $\psi$ we need a few definitions. Let $M$ be a
matching and let $O(M)$ and $C(M)$ be the set of openers and closers
of $M$, respectively. Write 
$$O(M) = O_1\cup\dots\cup O_k \;\text{ and }\; C(M) = C_1\cup\dots\cup C_{\ell}
$$ as disjoint unions of maximal intervals. Clearly, $k=\ell$; we denote this
number $\inter(M)$. As an example, for $M=\{(1,2),(3,5),(4,6)\}$ we
have $O(M) = [1,1] \cup [3,4]$, $C(M) = [2,2] \cup [5,6]$ and
$\inter(M)=2$.

We are now in a position to define the promised map from matchings to
matrices. Assume that $M$ is a matching and that its intervals of
openers and closers are ordered in the natural order.  Let $\psi(M) =
T$ where $T=(t_{ij})$ is an $\inter(M)\times\inter(M)$ matrix and
$$t_{ij} = |M \,\cap\, O_i\!\times\! C_j|.
$$ 
In other words, $t_{ij}$ is the number of arcs whose opener is in
$O_i$ and closer in $C_j$. For instance, the preimage of 
$\left(
\begin{smallmatrix}
  1 & 1 \\
  0 & 1
\end{smallmatrix}
\right)
$
under $\psi$ consists of the following 4 matchings:
$$
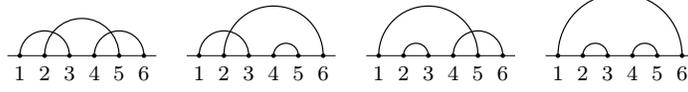
\begin{tikzpicture}[line width=.45pt, scale=0.33]
  \nodestyle
  \tikzstyle{opener} = []  
  \filldraw   (1,0) circle (1.9pt); \node[opener] at (1,-0.7) {1};
  \filldraw   (2,0) circle (1.9pt); \node[opener] at (2,-0.7) {2};
  \filldraw   (3,0) circle (1.9pt); \node at (3,-0.7) {3};
  \filldraw   (4,0) circle (1.9pt); \node[opener] at (4,-0.7) {4};
  \filldraw   (5,0) circle (1.9pt); \node at (5,-0.7) {5};
  \filldraw   (6,0) circle (1.9pt); \node at (6,-0.7) {6};
   \draw[very thin] (0.5,0) -- (6.500000,0);
  \draw (3,0) arc (0:180:1.0000);
  \draw (5,0) arc (0:180:1.5000);
  \draw (6,0) arc (0:180:1.0000);
\end{tikzpicture}\quad
\begin{tikzpicture}[line width=.45pt, scale=0.33]
  \nodestyle
  \tikzstyle{opener} = []
  \filldraw   (1,0) circle (1.9pt); \node[opener] at (1,-0.7) {1};
  \filldraw   (2,0) circle (1.9pt); \node[opener] at (2,-0.7) {2};
  \filldraw   (3,0) circle (1.9pt); \node at (3,-0.7) {3};
  \filldraw   (4,0) circle (1.9pt); \node[opener] at (4,-0.7) {4};
  \filldraw   (5,0) circle (1.9pt); \node at (5,-0.7) {5};
  \filldraw   (6,0) circle (1.9pt); \node at (6,-0.7) {6};
   \draw[very thin] (0.5,0) -- (6.500000,0);
  \draw (3,0) arc (0:180:1.0000);
  \draw (5,0) arc (0:180:0.5000);
  \draw (6,0) arc (0:180:2.0000);
\end{tikzpicture}\quad
\begin{tikzpicture}[line width=.45pt, scale=0.33]
  \nodestyle
  \tikzstyle{opener} = []
  \filldraw   (1,0) circle (1.9pt); \node[opener] at (1,-0.7) {1};
  \filldraw   (2,0) circle (1.9pt); \node[opener] at (2,-0.7) {2};
  \filldraw   (3,0) circle (1.9pt); \node at (3,-0.7) {3};
  \filldraw   (4,0) circle (1.9pt); \node[opener] at (4,-0.7) {4};
  \filldraw   (5,0) circle (1.9pt); \node at (5,-0.7) {5};
  \filldraw   (6,0) circle (1.9pt); \node at (6,-0.7) {6};
   \draw[very thin] (0.5,0) -- (6.500000,0);
  \draw (3,0) arc (0:180:0.5000);
  \draw (5,0) arc (0:180:2.0000);
  \draw (6,0) arc (0:180:1.0000);
\end{tikzpicture}\quad
\begin{tikzpicture}[line width=.45pt, scale=0.33]
  \nodestyle
  \tikzstyle{opener} = []  
  \filldraw   (1,0) circle (1.9pt); \node[opener] at (1,-0.7) {1};
  \filldraw   (2,0) circle (1.9pt); \node[opener] at (2,-0.7) {2};
  \filldraw   (3,0) circle (1.9pt); \node at (3,-0.7) {3};
  \filldraw   (4,0) circle (1.9pt); \node[opener] at (4,-0.7) {4};
  \filldraw   (5,0) circle (1.9pt); \node at (5,-0.7) {5};
  \filldraw   (6,0) circle (1.9pt); \node at (6,-0.7) {6};
   \draw[very thin] (0.5,0) -- (6.500000,0);
  \draw (3,0) arc (0:180:0.5000);
  \draw (5,0) arc (0:180:0.5000);
  \draw (6,0) arc (0:180:2.5000);
\end{tikzpicture}
$$ Note that of these matchings exactly one has no neighbor nestings
and exactly one has no neighbor crossings. We shall see that this is no
coincidence.

\begin{theorem}\label{thm:no-neighbor-nestings}
  When restricted to matchings of \/$[2n]$\! with no neighbor nestings 
  the function $\psi$, defined above, is a bijection onto $\T_n$.
\end{theorem}

Before we give the proof we need a lemma.

\begin{lemma}\label{L:structure} 
  Let $M$ be a matching. Assume that $O'$ is an interval of openers
  in $M$, and that $C'$ is an interval of closers in $M$. We have: 
  \begin{enumerate}
  \item\label{first} if \makebox[11.8ex][r]{$\lne(M)=0$} then any pair
    of arcs with openers in $O'$ are crossing;
  \item\label{second} if \makebox[11.8ex][r]{$\rne(M)=0$} then any pair
    of arcs with closers in $C'$ are crossing;
  \item\label{third} if \makebox[11.8ex][r]{$\lcr(M)=0$} then any pair
    of arcs with openers in $O'$ are nesting;
  \item\label{fourth} if \makebox[11.8ex][r]{$\rcr(M)=0$} then any pair
    of arcs with closers in $C'$ are nesting.
  \end{enumerate}
\end{lemma}

\begin{proof}
  We shall prove \eqref{first}. The remaining three statements are
  proved by similar arguments. Without loss of generality we may
  assume that $O'=\{o_i,o_{i+1},\dots,o_{i+j}\}$ is a maximal
  interval. If $M$ has no left-nestings then the arcs from $o_i$ and
  $o_{i+1}$ must cross. Similarly the arc from $o_{i+2}$ must cross
  the arc from $o_{i+1}$ and thus also the arc from $o_i$. So by an
  easy induction argument all the arcs must cross.
\end{proof}

\begin{proof}[Proof of Theorem~\ref{thm:no-neighbor-nestings}]
  Let $T\in\T_n$ be a $k\times k$ matrix. We shall show that there is
  a unique non-neighbor-nesting matchings $M$ of $[2n]$ such that
  $\psi(M)=T$.
  
  Let $r_i$ and $c_i$ be the sum of the entries in, respectively, row
  $i$ and column $i$ of $T$.  From the definition of $\psi$ it is
  clear that 
  $$O(M)=O_1\cup\dots\cup O_k\;\text{ and }\;C(M)=C_1\cup\dots\cup C_k,
  $$ where $O_1=[1,r_1]$, $C_1=[r_1+1,r_1+c_1]$, $O_2=[r_1+c_1+1,
    r_1+c_1+r_2]$, etc. Also, $M$ shall have $t_{ij}$ arcs from $O_i$
  to $C_j$. We must show how the arcs shall be drawn.

  Since $\lne(M)=0$ it follows from \eqref{first} of Lemma
  \ref{L:structure} that all arcs from $O_i$ cross. Thus we know which
  $t_{ij}$ openers, $X_{ij}$, in $O_i$ that will have arcs to
  $C_j$. Similarly, since $\rne(M)=0$ it follows by \eqref{second} of
  Lemma \ref{L:structure} that all arcs to $C_j$ cross. Thus we know
  which $t_{ij}$ of the closers, $Y_{ij}$, in $C_j$ that will have arcs
  from $O_i$. So, for every $1\le i\le j\le k$, $M$ must have crossing
  arcs from $X_{ij}$ to $Y_{ij}$.

  We have showed that there is exactly one $M$ that (by construction)
  satisfies $\psi(M)=T$, and so the theorem follows.
\end{proof}

\begin{theorem}\label{thm:no-neighbor-crossings}
  When restricted to matchings of \/$[2n]$\! with no neighbor crossings 
  the function $\psi$, defined above, is a bijection onto $\T_n$.
\end{theorem}

\begin{proof}
  The proof is essentially the same as for
  Theorem~\ref{thm:no-neighbor-nestings}. The difference is that
  here we use \eqref{third} and \eqref{fourth} of
  Lemma~\ref{L:structure}, instead of \eqref{first} and \eqref{second}.
\end{proof}

Below is an illustration of Theorems~\ref{thm:no-neighbor-nestings}
and \ref{thm:no-neighbor-crossings} for $n=3$. We have circled the
openers to make it easy to see the intervals.

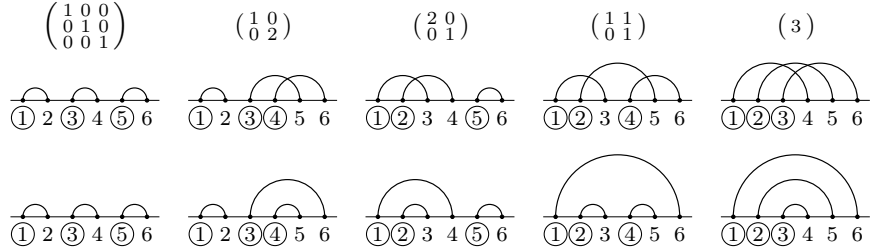
\renewcommand{\discstyle}{\tikzstyle{disc} = [
    circle,thin,fill=\cfilll,draw=black, minimum size=3.1pt, inner
    sep=0pt ] }
$$
\begin{array}{ccccc}
  \left(
  \begin{smallmatrix}
    1 & 0 & 0 \\
    0 & 1 & 0 \\
    0 & 0 & 1
  \end{smallmatrix}
  \right) &
  \left(
  \begin{smallmatrix}
    1 & 0 \\
    0 & 2
  \end{smallmatrix}
  \right) &
  \left(
  \begin{smallmatrix}
    2 & 0 \\
    0 & 1
  \end{smallmatrix}
  \right) &
  \left(
  \begin{smallmatrix}
    1 & 1 \\
    0 & 1
  \end{smallmatrix}
  \right) &
  \left(\hspace{-0.2pt}
  \begin{smallmatrix}
    3
  \end{smallmatrix}\hspace{-0.2pt}
  \right)\\[1.8ex]
  \begin{tikzpicture}[line width=.45pt, scale=0.33]
    \nodestyle
    \tikzstyle{opener} = [thin, circle, draw=black, inner sep=0.7pt]
  \filldraw (1,0) circle (1.9pt); \node[opener] at (1,-0.7) {1};
  \filldraw (2,0) circle (1.9pt); \node at (2,-0.7) {2};
  \filldraw (3,0) circle (1.9pt); \node[opener] at (3,-0.7) {3};
  \filldraw (4,0) circle (1.9pt); \node at (4,-0.7) {4};
  \filldraw (5,0) circle (1.9pt); \node[opener] at (5,-0.7) {5};
  \filldraw (6,0) circle (1.9pt); \node at (6,-0.7) {6};
   \draw[very thin] (0.5,0) -- (6.500000,0);
  \draw (2,0) arc (0:180:0.5000);
  \draw (4,0) arc (0:180:0.5000);
  \draw (6,0) arc (0:180:0.5000);
  \end{tikzpicture}&
  \begin{tikzpicture}[line width=.45pt, scale=0.33]
    \nodestyle
    \tikzstyle{opener} = [thin, circle, draw=black, inner sep=0.7pt]  
  \filldraw (1,0) circle (1.9pt); \node[opener] at (1,-0.7) {1};
  \filldraw (2,0) circle (1.9pt); \node at (2,-0.7) {2};
  \filldraw (3,0) circle (1.9pt); \node[opener] at (3,-0.7) {3};
  \filldraw (4,0) circle (1.9pt); \node[opener] at (4,-0.7) {4};
  \filldraw (5,0) circle (1.9pt); \node at (5,-0.7) {5};
  \filldraw (6,0) circle (1.9pt); \node at (6,-0.7) {6};
   \draw[very thin] (0.5,0) -- (6.500000,0);
  \draw (2,0) arc (0:180:0.5000);
  \draw (5,0) arc (0:180:1.0000);
  \draw (6,0) arc (0:180:1.0000);
  \end{tikzpicture}&
  \begin{tikzpicture}[line width=.45pt, scale=0.33]
    \nodestyle
    \tikzstyle{opener} = [thin, circle, draw=black, inner sep=0.7pt]
  \filldraw (1,0) circle (1.9pt); \node[opener] at (1,-0.7) {1};
  \filldraw (2,0) circle (1.9pt); \node[opener] at (2,-0.7) {2};
  \filldraw (3,0) circle (1.9pt); \node at (3,-0.7) {3};
  \filldraw (4,0) circle (1.9pt); \node at (4,-0.7) {4};
  \filldraw (5,0) circle (1.9pt); \node[opener] at (5,-0.7) {5};
  \filldraw (6,0) circle (1.9pt); \node at (6,-0.7) {6};
   \draw[very thin] (0.5,0) -- (6.500000,0);
  \draw (3,0) arc (0:180:1.0000);
  \draw (4,0) arc (0:180:1.0000);
  \draw (6,0) arc (0:180:0.5000);
  \end{tikzpicture}&
  \begin{tikzpicture}[line width=.45pt, scale=0.33]
    \nodestyle
    \tikzstyle{opener} = [thin, circle, draw=black, inner sep=0.7pt]  
  \filldraw (1,0) circle (1.9pt); \node[opener] at (1,-0.7) {1};
  \filldraw (2,0) circle (1.9pt); \node[opener] at (2,-0.7) {2};
  \filldraw (3,0) circle (1.9pt); \node at (3,-0.7) {3};
  \filldraw (4,0) circle (1.9pt); \node[opener] at (4,-0.7) {4};
  \filldraw (5,0) circle (1.9pt); \node at (5,-0.7) {5};
  \filldraw (6,0) circle (1.9pt); \node at (6,-0.7) {6};
   \draw[very thin] (0.5,0) -- (6.500000,0);
  \draw (3,0) arc (0:180:1.0000);
  \draw (5,0) arc (0:180:1.5000);
  \draw (6,0) arc (0:180:1.0000);
  \end{tikzpicture}&
  \begin{tikzpicture}[line width=.45pt, scale=0.33]
    \nodestyle
    \tikzstyle{opener} = [thin, circle, draw=black, inner sep=0.7pt]
  \filldraw (1,0) circle (1.9pt); \node[opener] at (1,-0.7) {1};
  \filldraw (2,0) circle (1.9pt); \node[opener] at (2,-0.7) {2};
  \filldraw (3,0) circle (1.9pt); \node[opener] at (3,-0.7) {3};
  \filldraw (4,0) circle (1.9pt); \node at (4,-0.7) {4};
  \filldraw (5,0) circle (1.9pt); \node at (5,-0.7) {5};
  \filldraw (6,0) circle (1.9pt); \node at (6,-0.7) {6};
   \draw[very thin] (0.5,0) -- (6.500000,0);
  \draw (4,0) arc (0:180:1.5000);
  \draw (5,0) arc (0:180:1.5000);
  \draw (6,0) arc (0:180:1.5000);
  \end{tikzpicture} \\[1ex]
  \begin{tikzpicture}[line width=.45pt, scale=0.33]
    \nodestyle
    \tikzstyle{opener} = [thin, circle, draw=black, inner sep=0.7pt]  
  \filldraw   (1,0) circle (1.9pt); \node[opener] at (1,-0.7) {1};
  \filldraw   (2,0) circle (1.9pt); \node at (2,-0.7) {2};
  \filldraw   (3,0) circle (1.9pt); \node[opener] at (3,-0.7) {3};
  \filldraw   (4,0) circle (1.9pt); \node at (4,-0.7) {4};
  \filldraw   (5,0) circle (1.9pt); \node[opener] at (5,-0.7) {5};
  \filldraw   (6,0) circle (1.9pt); \node at (6,-0.7) {6};
   \draw[very thin] (0.5,0) -- (6.500000,0);
  \draw (2,0) arc (0:180:0.5000);
  \draw (4,0) arc (0:180:0.5000);
  \draw (6,0) arc (0:180:0.5000);
  \end{tikzpicture} &
  \begin{tikzpicture}[line width=.45pt, scale=0.33]
    \nodestyle
    \tikzstyle{opener} = [thin, circle, draw=black, inner sep=0.7pt]
  \filldraw   (1,0) circle (1.9pt); \node[opener] at (1,-0.7) {1};
  \filldraw   (2,0) circle (1.9pt); \node at (2,-0.7) {2};
  \filldraw   (3,0) circle (1.9pt); \node[opener] at (3,-0.7) {3};
  \filldraw   (4,0) circle (1.9pt); \node[opener] at (4,-0.7) {4};
  \filldraw   (5,0) circle (1.9pt); \node at (5,-0.7) {5};
  \filldraw   (6,0) circle (1.9pt); \node at (6,-0.7) {6};
   \draw[very thin] (0.5,0) -- (6.500000,0);
  \draw (2,0) arc (0:180:0.5000);
  \draw (5,0) arc (0:180:0.5000);
  \draw (6,0) arc (0:180:1.5000);
  \end{tikzpicture} &
  \begin{tikzpicture}[line width=.45pt, scale=0.33]
    \nodestyle
    \tikzstyle{opener} = [thin, circle, draw=black, inner sep=0.7pt]  
  \filldraw   (1,0) circle (1.9pt); \node[opener] at (1,-0.7) {1};
  \filldraw   (2,0) circle (1.9pt); \node[opener] at (2,-0.7) {2};
  \filldraw   (3,0) circle (1.9pt); \node at (3,-0.7) {3};
  \filldraw   (4,0) circle (1.9pt); \node at (4,-0.7) {4};
  \filldraw   (5,0) circle (1.9pt); \node[opener] at (5,-0.7) {5};
  \filldraw   (6,0) circle (1.9pt); \node at (6,-0.7) {6};
   \draw[very thin] (0.5,0) -- (6.500000,0);
  \draw (3,0) arc (0:180:0.5000);
  \draw (4,0) arc (0:180:1.5000);
  \draw (6,0) arc (0:180:0.5000);
  \end{tikzpicture} &
  \begin{tikzpicture}[line width=.45pt, scale=0.33]
    \nodestyle
    \tikzstyle{opener} = [thin, circle, draw=black, inner sep=0.7pt]  
  \filldraw   (1,0) circle (1.9pt); \node[opener] at (1,-0.7) {1};
  \filldraw   (2,0) circle (1.9pt); \node[opener] at (2,-0.7) {2};
  \filldraw   (3,0) circle (1.9pt); \node at (3,-0.7) {3};
  \filldraw   (4,0) circle (1.9pt); \node[opener] at (4,-0.7) {4};
  \filldraw   (5,0) circle (1.9pt); \node at (5,-0.7) {5};
  \filldraw   (6,0) circle (1.9pt); \node at (6,-0.7) {6};
   \draw[very thin] (0.5,0) -- (6.500000,0);
  \draw (3,0) arc (0:180:0.5000);
  \draw (5,0) arc (0:180:0.5000);
  \draw (6,0) arc (0:180:2.5000);
  \end{tikzpicture} &
  \begin{tikzpicture}[line width=.45pt, scale=0.33]
    \nodestyle
    \tikzstyle{opener} = [thin, circle, draw=black, inner sep=0.7pt]  
  \filldraw   (1,0) circle (1.9pt); \node[opener] at (1,-0.7) {1};
  \filldraw   (2,0) circle (1.9pt); \node[opener] at (2,-0.7) {2};
  \filldraw   (3,0) circle (1.9pt); \node[opener] at (3,-0.7) {3};
  \filldraw   (4,0) circle (1.9pt); \node at (4,-0.7) {4};
  \filldraw   (5,0) circle (1.9pt); \node at (5,-0.7) {5};
  \filldraw   (6,0) circle (1.9pt); \node at (6,-0.7) {6};
   \draw[very thin] (0.5,0) -- (6.500000,0);
  \draw (4,0) arc (0:180:0.5000);
  \draw (5,0) arc (0:180:1.5000);
  \draw (6,0) arc (0:180:2.5000);
  \end{tikzpicture}
\end{array}
$$
\renewcommand{\discstyle}{\tikzstyle{disc} = 
  [ circle,thin,fill=\cfilll,draw=black, minimum size=\ns, inner sep=0pt ] }

We have now explained the hierarchy of nesting and crossing conditions
that we set out to explain in the beginning of this section. As we
pointed out, the bijections for the more general cases do not
specialize to give bijections between the smaller sets.  Indeed, if we
specialize the map $\psi$ to matchings with no nestings we get the
subset of matrices $(t_{ij})\in \T_n$ such that for all $i,j,x,y>0$,
at least one of $t_{i,j}$ and $t_{i-x,j+y}$ must be zero. The non-zero
entries in such a matrix will form a ``path'' with the entries as
vertices, which can be seen to be equivalent to a Motzkin path.  Thus,
the matrices just described are in bijection with Motzkin paths with
positive integer weights on the vertices of the path such that the sum
of the weights is $n$.  As an example, for $n=3$ we have these five
paths:
$$
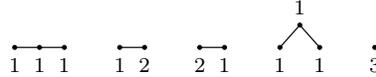
\begin{tikzpicture}[line width=.45pt, scale=0.33]
  \nodestyle
  \filldraw (1,0) circle (2.2pt); \node at (1,-0.7) {1};
  \filldraw (2,0) circle (2.2pt); \node at (2,-0.7) {1};
  \filldraw (3,0) circle (2.2pt); \node at (3,-0.7) {1};
  \draw (1,0) -- (2,0) -- (3,0);
\end{tikzpicture}\quad
\begin{tikzpicture}[line width=.45pt, scale=0.33]
  \nodestyle
  \filldraw (1,0) circle (2.2pt); \node at (1,-0.7) {1};
  \filldraw (2,0) circle (2.2pt); \node at (2,-0.7) {2};
  \draw (1,0) -- (2,0);
\end{tikzpicture}\quad
\begin{tikzpicture}[line width=.45pt, scale=0.33]
  \nodestyle
  \filldraw (1,0) circle (2.2pt); \node at (1,-0.7) {2};
  \filldraw (2,0) circle (2.2pt); \node at (2,-0.7) {1};
  \draw (1,0) -- (2,0);
\end{tikzpicture}\quad
\begin{tikzpicture}[line width=.45pt, scale=0.33]
  \nodestyle
  \filldraw (1.0, 0)   circle (2.2pt); \node at (1.0,-0.7) {1};
  \filldraw (1.8, 0.9) circle (2.2pt); \node at (1.8, 1.6) {1};
  \filldraw (2.6, 0)   circle (2.2pt); \node at (2.6,-0.7) {1};
  \draw (1.0, 0) -- (1.8, 0.9) -- (2.6, 0);
\end{tikzpicture}\quad
\begin{tikzpicture}[line width=.45pt, scale=0.33]
  \nodestyle
  \filldraw (1,0) circle (2.2pt); \node at (1,-0.7) {3};
\end{tikzpicture}\quad
$$
If we on the other hand specialize $\psi$ to matchings
with no crossings we get the somewhat odd constraint that for all
$i<i+x\le j<j+y$ at least one of $t_{i,j}$ and $t_{i+x,j+y}$ must be
zero.

\begin{corollary}
  The two subsets of $\T_n$ mentioned above are enumerated by the
  Catalan numbers.
\end{corollary}

Before we close this section we give one more result that is almost
for free given the map $\psi$. Let $\T^{01}_n\subset \T_n$ be the set
of zero-one matrices in $\T_n$. For instance,
$$\T^{01}_3 =
\left\{\,\left(
\begin{smallmatrix}
  1 & 1 \\
  0 & 1
\end{smallmatrix}
\right),\,
\left(
\begin{smallmatrix}
  1 & 0 & 0 \\
  0 & 1 & 0 \\
  0 & 0 & 1
\end{smallmatrix}
\right)\,\right\}.
$$
Dukes and Parviainen~\cite[\S4]{DP} showed that the matrices in
$\T^{01}_n$ correspond to those ascent sequences that have no two
equal consecutive entries. We offer the following proposition.

\begin{proposition}\label{prop:01-matrices}
  When restricted to matchings of \/$[2n]$\! with no left-nestings and
  no right-crossings the function $\psi$, defined above, is a
  bijection onto $\T^{01}_n$.
\end{proposition}

\begin{proof}
  The proof is very similar to the proofs of
  Theorems~\ref{thm:no-neighbor-nestings} and
  \ref{thm:no-neighbor-crossings} so we omit most of it. We do however
  make this key observation: Assume that $1\le i\le j\le
  \inter(M)$. From \eqref{first} and \eqref{fourth} of
  Lemma~\ref{L:structure} there can be at most one arc from the $O_i$
  to $O_j$.
\end{proof}

An important remark is that there exists a bijection by Sundaram
\cite{sund} (see also exercise 7.24 in \cite{St2}), via certains walks
in the Youngs lattice called oscillating tableaux, that uniformly
shows all three cases above. For readers familiar with this bijection
let us briefly explain why.  Let $M$ be a matching of $[2n]$ in which
$j$ and $j+1$ are two consecutive closers.  Let
$(\la^0,\dots,\la^{2n})$, where $\la^0=\la^{2n}=\emptyset$, be the
oscillating tableau corresponding to $M$. The assumption that $j$ and
$j+1$ are closers means that $\la^{j+1}\subset \la^{j}\subset
\la^{j-1}$.

Let $r_{j+1}$ be the row where $\la^{j+1}$ has fewer elements than
$\la^{j}$ and let $r_{j}$ be the row where $\la^{j}$ has fewer
elements than $\la^{j-1}$. Following the bumping paths of column
insertion one can argue that the arcs ending in $j$ and $j+1$ form a
right-crossing if and only if $r_{j}\le r_{j+1}$, and thus they form a
right-nesting if and only if $r_{j}>r_{j+1}$.  Now, consider the
involution $M^*$ obtained by conjugating each partition in
$(\la^0,\dots,\la^{2n})$ and then applying the inverse of Sundarams
bijection.  A moment of thought gives that the arcs ending in $j$ and
$j+1$ form a right-nesting in $M$ if and only if they form a
right-crossing in $M^*$.  Similarly, if $i$ and $i+1$ are two
consecutive openers of $M$, then $\la^{i-1}\subset\la^{i}\subset
\la^{i+1}$. This time let $r_x$ be the row in which $\la^{x}$ is
greater than $\la^{x-1}$. Then the arcs with openers $i$ and $i+1$
form a left-nesting if and only if $r_i<r_{i+1}$.  Hence $i$ and $i+1$
form a left-nesting in $M$ if and only if they form a left-crossing in
$M^*$.

This shows that the bijection in \cite{sund} may be used to explain
all three levels discussed here at once. It also shows that using the
above restrictions we get two different subets of all vacillating
tableaux enumerated by $n!$ and one subset, satisfying both
restrictions, that is enumerated by the Fishburn numbers.

\section{Ascent and descent correcting sequences}

On contemplating the picture 
$$
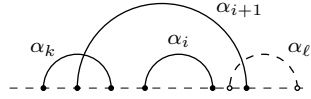
\begin{tikzpicture}[line width=.5pt, scale=0.45]  
  \draw[very thin, dashed] (0,0) -- (9,0);
  \node[font=\footnotesize] at (1,1.15) {$\al_k$};
  \node[font=\footnotesize] at (6.8,2.3) {$\al_{i+1}$};
  \node[font=\footnotesize] at (5,1.4) {$\al_i$};
  \node[font=\footnotesize] at (8.57,1.15) {$\al_{\ell}$};
  \draw (3,0) arc (0:180:1.0000);
  \draw (6,0) arc (0:180:1.0000);
  \draw (7,0) arc (0:180:2.5000);
  \draw[dashed] (8.5,0) arc (0:180:1.0000);
  \filldraw (1,0) circle (1.9pt);
  \filldraw (2,0) circle (1.9pt);
  \filldraw (3,0) circle (1.9pt);
  \filldraw (4,0) circle (1.9pt);
  \filldraw (6,0) circle (1.9pt);
  \filldraw[fill=white] (6.5,0) circle (1.9pt);
  \filldraw (7,0) circle (1.9pt);
  \filldraw[fill=white] (8.5,0) circle (1.9pt);
\end{tikzpicture}
$$ for a while, one realizes that condition~\eqref{1} in
Proposition~\ref{prop:2+2-free} is equivalent to
\begin{equation}\label{1var}
  i>_P k\text{ and } i+1\not>_P k
  \implies i=\pre(\ell) \text{ for some $\ell$ in $P$.} 
\end{equation}
Let a \emph{descent correcting sequence} be an inversion table
$(a_1,\dots,a_n)$ such that
$$a_{i}>a_{i+1}\implies a_\ell=i \,\text{ for some $\ell>i$}.
$$ That is, if there is a descent at position $i$ then this has to be
``corrected'' by the value $i$ occurring later in the
sequence. Condition~\eqref{1var} translates directly to the condition
for a descent correcting sequence, and thus we have the following
Proposition.

\newcommand{\hs}{\hspace*{-1pt}}
\begin{proposition} 
  There are exactly\/ $\fish_n$\hs\ descent correcting sequences of length
  $n$, where $\fish_n$\hs\ is the $n$th Fishburn number.
\end{proposition}

We may similarly use the map $f_{\nc}$ from matchings with no
left-crossings to inversion tables. We then get that the sequences
corresponding to matchings with no neighbor crossings are the
inversion tables $(a_1,\dots,a_n)$ such that
$$a_{i}<a_{i+1}\neq i+1 \implies a_\ell=i \,\text{ for some $\ell>i$}. 
$$
We call them \emph{ascent correcting sequences}. Using 
Theorem~\ref{thm:no-neighbor-crossings} we arrive at the following result.

\begin{proposition} 
  There are exactly\/ $\fish_n$\hs\ ascent correcting sequences of
  length $n$.
\end{proposition}

\section{Posets that are both factorial and dually factorial}

Note that being dually factorial entails the condition in
Proposition~\ref{prop:2+2-free}. So, under $h$, matchings
corresponding to dually factorial posets have no right-nestings. In
fact, they do not have any nestings at all. To see this, assume that
$M\in\N_n$ and let $\al_1,\dots,\al_n$ are its arcs ordered by
closer. Also, assume that $1\leq i<j\leq n$. Recall that the arcs
$\al_i$ and $\al_j$ form a nesting precisely when $\pre(i)>\pre(j)$,
which is equivalent to there being a $k<_Pi$ such that $k\not<_Pj$;
this cannot happen in a dually factorial poset. It is easy to see that
this argument works both ways, so $M=h(P)$ is non-nesting if and only
if $P$ is dually factorial. It is well known that non-nesting matchings
are counted by the Catalan numbers. See for instance
Stanley~\cite[Ex. 6.19uu]{St2}. One way to associate a given
non-nesting matching with a Dyck path is to map its openers to
up-steps and its closers to down-steps.

\begin{proposition}
  There are exactly $C_n=\binom{2n}{n}/(n+1)$ posets on $[n]$ that are
  both factorial and dually factorial.
\end{proposition}

Let us mention an alternative way to prove this proposition. Bellow is
the smallest example of a factorial poset that is not dually factorial
but satisfies the condition of Proposition~\ref{prop:2+2-free}:
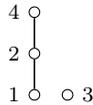
$$
\begin{tikzpicture}[line width=0.6pt, scale=0.55]
  \style
  \node [disc] (r1)   at ( 0,0) {};
  \node [disc] (r11)  at ( 0,-1) {};
  \node [disc] (r111) at ( 0,-2) {};
  \node [disc] (r2)   at ( 0.8,-2) {};
  \draw (r1) node[left=2pt] {4} -- (r11) node[left=2pt] {2} 
         -- (r111) node[left=2pt] {1}; 
  \draw (r2) node[right=2pt] {3}; 
\end{tikzpicture}
$$ As stated by Proposition~\ref{prop:factorial-posets-are-2+2-free},
factorial posets are $\tpt$-free; those that, in addition, are dually
factorial are $({\bf 3}+{\bf 1})$-free. 

\begin{proposition}
  If $P$ is a factorial poset satisfying~\eqref{1} from
  Proposition~\ref{prop:2+2-free}, then $P$ is dually factorial if and
  only if $P$ is $({\bf 3}+{\bf 1})$-free.
\end{proposition}

\begin{proof}
  For factorial posets $P$ on less than $4$ elements the result is
  trivial. We shall assume that $P$ has at least $4$ elements and
  prove the contra-positive statement: $P$ is not dually factorial if
  and only if $P$ contains an induced subposet isomorphic to ${\bf
    3}+{\bf 1}$.  Assume that the elements $x<_P y<_P z$ and $w$ form
  an induced subposet of $P$ that is isomorphic to ${\bf 3}+{\bf
    1}$. Since $P$ is factorial we cannot have $w < y <_P z$ and $w
  \not<_P z$; thus $w>y$. Then $w>y>_P x$ and $w\not >_P x$, so $P$ is
  not dually factorial. 

  Conversely, assume that $P$ is not dually factorial. Then there
  exists $x$, $y$ and $w$ with $w>y>_P x$ but $w\not >_P x$.  Assume
  further that $y$ is maximal and then $w$ minimal with this property.
  We claim that $w=y+1$. If not, then $y+1>_P x$ by the minimality of
  $w$. Using transitivity and $w\not >_P x$ this implies $w\not >_P
  y+1$, which contradicts the maximality of $y$.  Since $w=y+1$ and
  $\pre(y)>\pre(w)$, it follows from property~\eqref{1} that
  $\suc(y)>\suc(w)$. Thus there exists a $z$ such that $z>_P y$ and
  $z\not >_P w$, and so the induced subposet on $\{x,y,z,w\}$ is
  isomorphic to ${\bf 3}+{\bf 1}$.
\end{proof}

Since posets that are both factorial and dually factorial have a
unique labeling we can regard them as unlabeled. Further, unlabeled
posets that are both $\tpt$- and $({\bf 3}+{\bf 1})$-free (also called
semiorders) are known to be enumerated by the Catalan numbers;
see~\cite[Ex. 6.19ddd]{St2} and \cite{WiFr}.

\section{Statistics and equidistributions}

One question we shall consider in this section is what statistics are
respected by the bijections $f$, $g$ and $h$. For reference, we list
the size $3$ matchings, inversion tables, permutations and posets that
correspond to each other under those bijections:
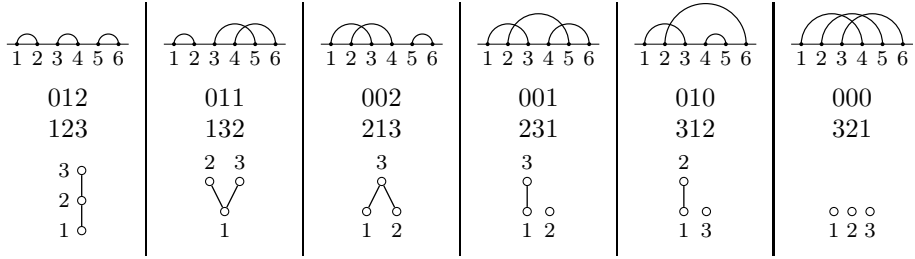
\renewcommand{\discstyle}{\tikzstyle{disc} = [
    circle,thin,fill=\cfilll,draw=black, minimum size=3.1pt, inner
    sep=0pt ] }
$$
\begin{array}{c|c|c|c|c|c}
  \begin{tikzpicture}[line width=.45pt, scale=0.27]
    \nodestyle
  
  \filldraw   (1,0) circle (1.9pt); \node at (1,-0.7) {1};
  \filldraw   (2,0) circle (1.9pt); \node at (2,-0.7) {2};
  \filldraw   (3,0) circle (1.9pt); \node at (3,-0.7) {3};
  \filldraw   (4,0) circle (1.9pt); \node at (4,-0.7) {4};
  \filldraw   (5,0) circle (1.9pt); \node at (5,-0.7) {5};
  \filldraw   (6,0) circle (1.9pt); \node at (6,-0.7) {6};
   \draw[very thin] (0.5,0) -- (6.500000,0);
  \draw (2,0) arc (0:180:0.5000);
  \draw (4,0) arc (0:180:0.5000);
  \draw (6,0) arc (0:180:0.5000);
  \end{tikzpicture}&
  \begin{tikzpicture}[line width=.45pt, scale=0.27]
    \nodestyle
  
  \filldraw   (1,0) circle (1.9pt); \node at (1,-0.7) {1};
  \filldraw   (2,0) circle (1.9pt); \node at (2,-0.7) {2};
  \filldraw   (3,0) circle (1.9pt); \node at (3,-0.7) {3};
  \filldraw   (4,0) circle (1.9pt); \node at (4,-0.7) {4};
  \filldraw   (5,0) circle (1.9pt); \node at (5,-0.7) {5};
  \filldraw   (6,0) circle (1.9pt); \node at (6,-0.7) {6};
   \draw[very thin] (0.5,0) -- (6.500000,0);
  \draw (2,0) arc (0:180:0.5000);
  \draw (5,0) arc (0:180:1.0000);
  \draw (6,0) arc (0:180:1.0000);
  \end{tikzpicture}&
  \begin{tikzpicture}[line width=.45pt, scale=0.27]
    \nodestyle
  
  \filldraw   (1,0) circle (1.9pt); \node at (1,-0.7) {1};
  \filldraw   (2,0) circle (1.9pt); \node at (2,-0.7) {2};
  \filldraw   (3,0) circle (1.9pt); \node at (3,-0.7) {3};
  \filldraw   (4,0) circle (1.9pt); \node at (4,-0.7) {4};
  \filldraw   (5,0) circle (1.9pt); \node at (5,-0.7) {5};
  \filldraw   (6,0) circle (1.9pt); \node at (6,-0.7) {6};
   \draw[very thin] (0.5,0) -- (6.500000,0);
  \draw (3,0) arc (0:180:1.0000);
  \draw (4,0) arc (0:180:1.0000);
  \draw (6,0) arc (0:180:0.5000);
  \end{tikzpicture}&
  \begin{tikzpicture}[line width=.45pt, scale=0.27]
    \nodestyle
  
  \filldraw   (1,0) circle (1.9pt); \node at (1,-0.7) {1};
  \filldraw   (2,0) circle (1.9pt); \node at (2,-0.7) {2};
  \filldraw   (3,0) circle (1.9pt); \node at (3,-0.7) {3};
  \filldraw   (4,0) circle (1.9pt); \node at (4,-0.7) {4};
  \filldraw   (5,0) circle (1.9pt); \node at (5,-0.7) {5};
  \filldraw   (6,0) circle (1.9pt); \node at (6,-0.7) {6};
   \draw[very thin] (0.5,0) -- (6.500000,0);
  \draw (3,0) arc (0:180:1.0000);
  \draw (5,0) arc (0:180:1.5000);
  \draw (6,0) arc (0:180:1.0000);
  \end{tikzpicture}&
  \begin{tikzpicture}[line width=.45pt, scale=0.27]
    \nodestyle
  
  \filldraw   (1,0) circle (1.9pt); \node at (1,-0.7) {1};
  \filldraw   (2,0) circle (1.9pt); \node at (2,-0.7) {2};
  \filldraw   (3,0) circle (1.9pt); \node at (3,-0.7) {3};
  \filldraw   (4,0) circle (1.9pt); \node at (4,-0.7) {4};
  \filldraw   (5,0) circle (1.9pt); \node at (5,-0.7) {5};
  \filldraw   (6,0) circle (1.9pt); \node at (6,-0.7) {6};
   \draw[very thin] (0.5,0) -- (6.500000,0);
  \draw (3,0) arc (0:180:1.0000);
  \draw (5,0) arc (0:180:0.5000);
  \draw (6,0) arc (0:180:2.0000);
  \end{tikzpicture}&  
  \begin{tikzpicture}[line width=.45pt, scale=0.27]
    \nodestyle
  
  \filldraw   (1,0) circle (1.9pt); \node at (1,-0.7) {1};
  \filldraw   (2,0) circle (1.9pt); \node at (2,-0.7) {2};
  \filldraw   (3,0) circle (1.9pt); \node at (3,-0.7) {3};
  \filldraw   (4,0) circle (1.9pt); \node at (4,-0.7) {4};
  \filldraw   (5,0) circle (1.9pt); \node at (5,-0.7) {5};
  \filldraw   (6,0) circle (1.9pt); \node at (6,-0.7) {6};
   \draw[very thin] (0.5,0) -- (6.500000,0);
  \draw (4,0) arc (0:180:1.5000);
  \draw (5,0) arc (0:180:1.5000);
  \draw (6,0) arc (0:180:1.5000);
  \end{tikzpicture}\\
  012 & 011 & 002 & 001 & 010 & 000 \\
  123 & 132 & 213 & 231 & 312 & 321 \\
  \pA{0.5pt}{0.4} & 
  \pE{0.5pt}{0.4} & 
  \pD{0.5pt}{0.4} & 
  \pC{0.5pt}{0.4} & 
  \pB{0.5pt}{0.4} & 
  \pF{0.5pt}{0.4}
\end{array}
$$
\renewcommand{\discstyle}{\tikzstyle{disc} = 
  [ circle,thin,fill=\cfilll,draw=black, minimum size=\ns, inner sep=0pt ] }

There are several well known ways of translating between permutations
and inversion tables. Here we have chosen the following way: Given
$\pi\in\sym_n$, we build the corresponding inversion table $w$ from
right to left. The right most letter of $w$ is $\pi^{-1}(n)-1$. The
remaining letters of $w$ are obtained by repeating this procedure on
the length $n-1$ permutation that results from $\pi$ by deleting $n$.

We shall now define the relevant statistics, and we start with
statistics on posets. The ordinal sum~\cite[\S3.2]{St1} of two posets
$P$ and $Q$ is the poset $P\oplus Q$ on the union $P\cup Q$ such that
$x\leq y$ in $P\oplus Q$ if $x\leq_P y$ or $x\leq_Q y$, or $x\in P$
and $y\in Q$. Let us say that $P$ has $k$ \emph{components}, and write
$\comp(P)=k$, if $P$ is the ordinal sum of $k$, but not $k+1$,
nonempty posets. The number of minimal elements of a poset $P$ is
denoted $\min(P)$. The number of levels of $P$---in other words, the
number of distinct predecessor sets in $P$---is denoted $\lev(P)$. A
pair of elements $x$ and $y$ in $P$ are said to be \emph{incomparable}
if $x\not\leq_Py$ and $y\not\leq_Px$. The number of incomparable pairs
in $P$ we denote by $\ip(P)$.

Let $\pi$ be a permutation. An \emph{ascent} in $\pi$ is a letter
followed by a larger letter; a \emph{descent} in $\pi$ is a letter
followed by a smaller letter. The number of ascents and descents are
denoted $\asc(\pi)$ and $\des(\pi)$, respectively. An inversion is a
pair $i<j$ such that $\pi(i)>\pi(j)$. The number of inversions is
denoted $\inv(\pi)$.  A \emph{left-to-right minimum} of $\pi$ is a
letter with no smaller letter to the left of it; the number of
left-to-right minima is denoted $\lmin(\pi)$.  The statistics
\emph{right-to-left minima} ($\rmin$), \emph{left-to-right maxima}
($\lmax$), and \emph{right-to-left maxima} ($\rmax$) are defined
similarly. For permutations $\pi$ and $\sigma$, let
$\pi\oplus\sigma=\pi\sigma'$, where $\sigma'$ is obtained from
$\sigma$ by adding $|\pi|$ to each of its letters, and juxtaposition
denotes concatenation. We say that $\pi$ has $k$ components, and write
$\comp(\pi)=k$, if $\pi$ is the sum of $k$, but not $k+1$, nonempty
permutations. Let $\dent(\pi)$ denote the number of distinct entries
of the inversion table associated with $\pi$.

For $M$ a matching on $[2m]$ and $N$ a matching on $[2n]$, let
$M\oplus N=M\cup N'$, where $N'$ is the matching on $[2m+1,2m+2n]$
obtained from $N$ by adding $2m$ to all of its openers and
closers. Let us say that $M$ has $k$ \emph{components}, and write
$\comp(M)=k$, if $M$ is the sum of $k$, but not $k+1$, nonempty
matchings. Let $\min(M)=j-1$ where $j$ is the smallest closer of $M$;
for a matching with no left nestings, $j$ is the closer of the arc with opener
$1$. Let $\last(M)$ be the number of closers that are smaller than the
opener of the last arc. Recall from Section~\ref{sec:cross-vs-nest}
that $\inter(M)$ denotes the number of intervals in the list of
openers of $M$.  Let us assume that $k$ is the closer of some arc of
$M$, and let $\al=(i,j)$ be another arc of $M$. If $i<k<j$ we say that
$k$ is \emph{embraced} by $\al$, and by $\emb(M)$ we denote the 
number of pairs $(k,\al)$ in $M$ such that the closer $k$ is embraced by $\al$.

\begin{proposition}
  Let $f$ and $g$ be as in the proofs of
  Theorems~\ref{thm:no-left-nestings} and \ref{thm:factorial-poset}.
  Let $P$ be a factorial poset on $[n]$. Let $w=g(P)$ and $M=f(w)$ be
  the corresponding inversion table and matching, respectively. Let
  $\pi$ be the permutation corresponding to $w$. Then
  $$
  \begin{array}{rlllllll}
    &(\!\!&\comp(P),&\min(P),&\pre(n),&\lev(P),&\ip(P)&\!\!)\;=\\
    &(\!\!&\comp(\pi),&\lmin(\pi),&\pi^{-1}(n)-1,&\dent(\pi),&\inv(\pi)&\!\!)\;=\\
    &(\!\!&\comp(M),&\min(M),&\last(M),&\inter(M),&\emb(M)&\!\!)
  \end{array}
  $$
\end{proposition}
\begin{proof}
  For inversion tables $u$ and $v$, let $u\oplus v=uv'$, where $v'$ is
  obtained from $v$ by adding $1+\max(u)$ to each of its letters, and
  juxtaposition denotes concatenation. It is easy to see that if
  $\sigma$ and $\tau$ are the permutations corresponding to $u$ and
  $v$, respectively, then the permutation corresponding to $u\oplus v$
  is $\sigma\oplus\tau$. Also, $f(u\oplus v)= f(u)\oplus f(v)$ and,
  for factorial posets $Q$ and $R$, $g(Q\oplus R) =g(Q)\oplus
  g(R)$. It follows that $\comp(P)=\comp(\pi)=\comp(M)$.

  The numbers $\min(P)$, $\lmin(\pi)$ and $\min(M)$ are all equal to
  the number of zeros in $w$. It is plain that $\pre(n)=\pi^{-1}(n)-1=\last(M)$
  and $\lev(P)=\dent(\pi)=\inter(M)$. It remains to show that
  $\ip(P)=\inv(\pi)=\emb(M)$. Note that
  \begin{align*}
    \ip(P)    &= \textstyle{\binom{n}{2}} 
    -\#\{(i,j): i<_Pj \},\\
    \inv(\pi) &= \textstyle{\binom{n}{2}} 
    -\#\{(i,j): i<j \text{ and } \pi(i)<\pi(j)\},
\intertext{ and, if $\al_1,\dots,\al_n$ are $M$'s arcs ordered by closer, then}
    \emb(M)   &= \textstyle{\binom{n}{2}}
    -\#\{(i,j): i<j\text{ and closer of $\al_i$ $<$ opener of $\al_j$}\}.
  \end{align*}
  It follows that $\ip(P)$, $\inv(\pi)$ and $\emb(M)$ are all equal to
  $\binom{n}{2}$ minus the sum of entries in the inversion table $w$.
\end{proof}

Let us note a few direct consequences of the above proposition.

\begin{corollary}
  The statistic $\ip$ is Mahonian on $\F_n$. That is, it has
  the same distribution as $\inv$ on $\sym_n$. Also, the statistic
  $\emb$ is Mahonian on $\N_n$.
\end{corollary}

\begin{corollary}
  The statistic $\lev$ is Eulerian on the set $\F_n$. That is, it has
  the same distribution as $\des$ on $\sym_n$. Also, the statistic
  $\inter$ is Eulerian on $\N_n$.
\end{corollary}

\begin{proof}
  It suffices to show that the statistic $\dent$ is Eulerian. The
  following proof is due to Emeric Deutsch (personal communication,
  May 2009). Let $d(n,k)$ be the number of inversion tables of length
  $n$ with $k$ distinct entries. Clearly, $d(n,0)=0$ for $n>0$ and
  $d(n,k)=0$ for $k>n$. We shall show that, for $0<k\leq n$,
  $$d(n,k)=kd(n-1,k) + (n-k+1)d(n-1,k-1).
  $$ (This recursion characterizes the Eulerian numbers.) Inversion
  tables of length $n$ with $k$ distinct entries fall into two
  disjoint classes: Those whose last entry is equal to at least one of
  the preceding $n-1$ entries; there are $kd(n-1,k)$ such inversion
  tables. Those whose last entry is different from the preceding $n-1$
  entries; there are $(n-(k-1))d(n-1,k-1)$ such inversion tables.
\end{proof}

Recall that $\lne(M)$ and $\rne(M)$ denote the number of left- and
right-nestings, respectively. Let $\lcr(M)$ and $\rcr(M)$ denote the
number of left- and right-crossings, respectively. The bijections
$f:\I_n\to\N_n$ and $g:\F_n\to\I_n$ that we have presented do not
specialize to the bijections presented by Bousquet-M\'elou et
al.~\cite{BCDK}. If one were to find bijections that do specialize in
the desired way, then one could also hope to prove the following
conjecture. Here we view $p=\pattern$ as a function counting the
occurrences of the pattern $p$. Also, for posets $P$, we define
$$\rne(P) = \#\{\,x\in P: \pre(x) > \pre(x+1)\;\,\text{and}\;\suc(x)=\suc(x+1)\,\}.
$$ In other words, $\rne(P)$ is the number of violations of
property~\eqref{1} of Proposition~\ref{prop:2+2-free}.

\begin{conjecture}\label{conj:distribution-of-rne}
  These three triples of statistics are equidistributed.
  $$
  \begin{array}{rlllll}    
    (\!\! & \rne, & \comp, & \min  & \!\!) &\text{on $\F_n$,} \\
    (\!\! & p,    & \comp, & \lmin & \!\!) &\text{on $\sym_n$,} \\
    (\!\! & \rne, & \comp, & \min  & \!\!) &\text{on $\N_n$.}
  \end{array}
  $$
\end{conjecture}

We also conjecture these additional equidistributions:

\begin{conjecture}\label{conj:distribution-of-rne-variant}
  These three triples of statistics are equidistributed.
  $$
  \begin{array}{rlllll}
    (\!\! & \rne, & \min,  & \lev-1,   & \!\!) &\text{on $\F_n$,} \\
    (\!\! & p,    & \lmax, & \des,     & \!\!) &\text{on $\sym_n$,} \\
    (\!\! & \rne, & \min,  & \inter-1, & \!\!) &\text{on $\N_n$.}
  \end{array}
  $$
\end{conjecture}

Conjectures~\ref{conj:distribution-of-rne} and
\ref{conj:distribution-of-rne-variant} have been checked by computer
for $n\leq 7$.

\section{Two additional conjectures and a generalization}

\begin{conjecture}\label{conj:2-left}
  Assume that $i<j<k<\ell$. Let us say that the arcs $(i,\ell)$ and
  $(j,k)$ are $m$-left-nesting if $j-i\leq m$. Note that a
  $1$-left-nesting is the same as a left-nesting. This conjecture
  claims that among all the matchings on $[2n]$ there are exactly
  $\fish_n$ that have no $2$-left-nestings.
\end{conjecture}

\begin{conjecture}\label{conj:eulerian}
  The distribution of $\lne$ over the set of all matchings on $[2n]$
  is given by the ``Second-order Eulerian triangle'', entry A008517 in
  OEIS~\cite{sloane}.
\end{conjecture}

Conjectures~\ref{conj:2-left} and
\ref{conj:eulerian} have been checked by computer
for $n\leq 7$.\medskip

\noindent
{\bf Note added in proof:} Paul Levande has found proofs for
Conjectures \ref{conj:2-left} and \ref{conj:eulerian}. See his
preprint arXiv:1006.3013.

\begin{problem}
  Consider the following generalization of factorial posets. Let
  $P$ and $Q$ be labeled posets on $[n]$ such that 
  $i<_P j\implies i<_Qj$. If, in addition, 
  $$ i<_Qj<_Pk\implies i<_P k
  $$ then we say that $P$ is \emph{$Q$-factorial}. Note that
  $\n$-factorial coincides with factorial, where $\n$ is the
  $n$-chain. Note also that $Q$ itself is always a $Q$-factorial poset
  and it is the only one if $Q$ is an antichain. Is this
  generalization useful? How many $Q$-factorial posets are there?
\end{problem}

\section*{Acknowledgment}

Thanks to Mark Dukes for helpful suggestions on the presentation.


\begin{thebibliography}{99}

\bibitem{bogart} K.~P. Bogart, An obvious proof of Fishburn's interval
  order theorem, {\em Discrete Mathematics} 118 no. 1-3 (1993)
  239--242.
		
\bibitem{BCDK} M.~Bousquet-M\'elou, A.~Claesson, M.~Dukes and
  S.~Kitaev, $(2+2)$-free posets, ascent sequences and pattern
  avoiding permutations, {\it Journal of Combinatorial Theory Series
    A} 117 (2010) 884--909.

\bibitem{deM} A.~de M\'edicis, X.~G.~Viennot, Moments des
  $q$-polyn\^{o}mes de Laguerre et la bijection de Foata-Zeilberger,
  {\it Advances in Applied Mathematics}, {\bf 15} (1994), 262--304.

\bibitem{deSC} M. de Sainte-Catherine, Couplage et Pfaffiens en
  combinatoire, physique et informatique, Th\`{e}se du 3me cycle,
  Universit\'e de Bordeaux I, 1983.

\bibitem{DP} M.~Dukes, R.~Parviainen,
  Ascent sequences and upper triangular matrices containing non-negative integers,
  {\it Electronic Journal of Combinatorics} {\bf 17}(1) (2010), \#R53 (16pp).

\bibitem{fishburn-book} P. C. Fishburn, {\em{Interval Graphs and Interval
    Orders}}, Wiley, New York, 1985.

\bibitem{fishburn} P. C. Fishburn, Intransitive indifference in
  preference theory: a survey, {\em Operational Research} {\bf 18} (1970)
  207--208.

\bibitem{fishburn2} P. C. Fishburn, Intransitive indifference with
  unequal indifference intervals, {\em{Journal of Mathematical
      Psychology}} {\bf 7} (1970) 144--149.

\bibitem{KaZe} A.~Kasraoui, J.~Zeng, Distribution of crossings,
  nestings and alignments of two edges in matchings and partitions,
  {\em Electronic Journal of Combinatorics} {\bf 13} no. 1 (2006) 12
  pp.

\bibitem{sloane} S. Plouffe and N. J. A. Sloane, The  Encyclopedia of
  Integer Sequences, Academic Press Inc.,
  San Diego, 1995.  Electronic version available at {\tt
    www.research.att.com/}$\!\sim${\tt njas/sequences/}.

\bibitem{St1} R.~P. Stanley, {\em Enumerative combinatorics {V}ol. 1},
  volume~49 of {\em Cambridge Studies in Advanced Mathematics},
  Cambridge University Press, Cambridge, 1997.

\bibitem{St2} R.~P. Stanley, {\em Enumerative combinatorics {V}ol. 2},
   volume~62 of {\em Cambridge Studies in Advanced Mathematics},
   Cambridge University Press, Cambridge, 1999.

\bibitem{stoim} A. Stoimenow, Enumeration of chord diagrams and an
  upper bound for Vassiliev invariants, {\em J. Knot Theory
    Ramifications} {\bf 7} no. 1 (1998) 93--114.

\bibitem{sund} S.~Sundaram, The Cauchy identity for Sp($2n$), {\em
  Journal of Combinatorial Theory, Series A} {\bf 53}(2) (1990)
  209--238

\bibitem{WiFr} R.L. Wine, J.E. Freund, On the enumeration of decision
  patterns involving n means, Annals of Mathematical Statistics {\bf
    28} (1957) 256--259.

\bibitem{zagier} D. Zagier, Vassiliev invariants and a strange
  identity related to the Dedeking eta-function, {\em Topology}, {\bf
    40} (2001) 945--960.


\end{thebibliography}
\end{document}